\documentclass{amsart}[12pt]
\usepackage{amssymb,amsmath}
\usepackage{enumerate}
\usepackage[english]{babel}
\usepackage{color}

\newcommand{\Om} {\Omega}
\newcommand {\ep} {\epsilon}
\newcommand {\om} {\omega}
\newcommand {\ii} {\infty}

\newcommand {\al} {\alpha}
\newcommand {\bt} {\beta}
\newcommand {\lb} {\lambda}

\newcommand {\ol} {\overline}
\newcommand {\sm} {\setminus}
\newcommand {\su} {\subset}

\newcommand {\cal} {\mathcal}

\newtheorem{teo}{Theorem}[section]
\newtheorem{pro}{Proposition}[section]
\newtheorem{cor}{Corollary}[section]

\newtheorem{lm}{Lemma}[section]

\theoremstyle{definition}
\newtheorem{rem}{Remark}[section]
\newtheorem{df}{Definition}[section]

\title{Pointwise ergodic theorems in symmetric spaces of measurable functions}
\keywords{infinite measure, Dunford-Schwartz  pointwise ergodic theorem,
fully symmetric Banach space, Return Times theorem, bounded Besicovitch sequence}
\subjclass[2010]{47A35(primary), 37A30(secondary)}
\begin{document}
\date{December 15, 2016}

\begin{abstract}
For a Dunford-Schwartz operator in a fully symmetric space of measurable functions of an arbitrary measure space,
we prove pointwise convergence of the conventional and weighted  ergodic averages.
\end{abstract}

\author{VLADIMIR CHILIN, \ DO\u GAN \c C\" OMEZ, \ SEMYON LITVINOV}
\address{The National University of Uzbekistan, Tashkent, Uzbekistan}
\email{vladimirchil@gmail.com; chilin@ucd.uz}
\address{North Dakota State University, P.O.Box 6050, Fargo, ND, 58108, USA}
\email{dogan.comez@ndsu.edu}
\address{Pennsylvania State University \\ 76 University Drive \\ Hazleton, PA 18202, USA}
\email{snl2@psu.edu}

\maketitle

\section{Introduction}

 The celebrated Dunford-Schwartz and Wiener-Wintner-type ergodic theorems are two of the major
themes of ergodic theory.  Due to their fundamental roles, these theorems have been revisited ever since their first appearance.  For instance, A. Garcia \cite{ga} gave an elegant self-contained proof of Dunford-Schwartz Theorem for $L^1- L^{\infty}-$contractions, and I. Assani \cite{as0,as} extended Bourgain's Return Times theorem to $\sigma-$finite setting. 
In this article, among other results, we extend Dunford-Schwartz and Wiener-Wintner ergodic theorems to  fully symmetric spaces of measurable functions.

We begin by showing, in Section 3, that if one works in a space of real valued measurable functions, then the class of absolute
contractions coincides with the class of Dunford-Schwartz operators, hence one can assume without loss of generality
that the linear operator in question
contracts $L^\ii-$norm. This helps us derive, in Section 4,  Dunford-Schwartz pointwise
ergodic theorem in any fully symmetric space of
functions $E$ in an infinite measure space $\Om$ such that the characteristic function $\chi_\Om \notin E$. Note that, as
it is shown in Section 2 of the article, the class of such spaces $E$ is significantly wider than the class of
$L^p-spaces$, $1\leq p<\ii$.

Section 5 is devoted to extension of weighted Dunford-Schwartz-type ergodic theorems to fully symmetric spaces $E$ with
 $\chi_\Om \notin E$.

In the last, Section 6, of the article we utilize Return Times theorem for $\sigma-$finite measure
to show that Wiener-Wintner ergodic theorem holds in any fully symmetric space $E$ such that $\chi_\Om \notin E$
and with the set sequences $\{ \lb^k\}$, $\lb\in \{ z\in \mathbb C: \ |z|=1\}$, extended to the set all bounded Besicovitch sequences.

\section{Preliminaries}
Let $(\Om,\mu)$ be a complete measure space.
Denote  by $\cal L^0$ ($\cal L_h^0$) the linear space of equivalence classes of almost everywhere (a.e.) finite
complex (respectively, real) valued measurable functions on $\Om$. Let $\chi_E$ be the characteristic function of a set $E\su \Om$.
Denote $\mathbf 1 = \chi_\Om$.
Given $1\leq p\leq \ii$, let $\cal L^p\su \cal L^0$ be
the $L^p-$space equipped with the standard norm $\| \cdot \|_p$.

Assume that $(\Om,\mu)$ is $\sigma$-finite. If $f \in \cal L^1 + \cal L^{\ii}$, then
a {\it non-increasing rearrangement} of $f$ is defined as
$$
\mu_t(f)=\inf \{\lb>0: \ \mu\{|f| > \lb\} \leq t\}, \ \  t>0
$$
(see \cite[Ch.II, \S 2]{kps}).

A Banach space $(E, \| \cdot \|_E)\su \cal L^1 + \cal L^{\ii}$ is called {\it symmetric (fully symmetric)} if
$$
f \in E, \ g \in \cal L^1 + \cal L^{\ii}, \ \mu_t(g)\leq \mu_t(f) \ \ \forall \ t>0
$$
(respectively,
$$
f \in E, \ g \in \cal L^1 + \cal L^{\ii}, \ \int \limits_0^s\mu_t(g)dt\leq  \int \limits_0^s\mu_t(f)dt \ \ \forall \ s>0
 \ (\text {writing } \  g \prec\prec f))
$$
implies that $g \in E$ and $\| g\|_E\leq \| f\|_E$.

Simple examples of fully symmetric spaces are $\cal L^1\cap \cal L^{\infty}$  with the norm
$$
\|f\|_{\cal L^1\cap \cal L^{\infty}}=\max \left \{ \|f\|_1, \|f\|_{\infty} \right\}
$$
and $\cal L^1 + \cal L^{\infty}$ with the norm
$$
\|f\|_{\cal L^1 + \cal L^{\infty}}=\inf \left \{ \|g\|_1+ \|h\|_{\infty}: \ f = g + h, \ g \in \cal L^1, \ h \in \cal L^{\infty} \right \}=
\int_0^1 \mu_t(f) dt
$$
(see \cite[Ch. II, \S 4]{kps}).

Denote by $E_+$ the set all nonnegative functions from a symmetric space $E$.
A  symmetric space $(E, \| \cdot \|_E)$ is said to possess
{\it Fatou property} if conditions
$$
\{f_{n}\}\su E_+, \ \ f_{n}\leq f_{n+1}  \ \forall \ n, \ \ \sup_n \| f_{n}\|_E<\ii
$$
imply that there exists $f=\sup \limits_{n}f_{n}\in E_+$ and $\| f\|_E=\sup \limits_{n} \| f_{n}\|_E$.

It is know that if $E = E^{\times\times}$, where
$$
E^\times = \{g \in \cal L^1 + \cal L^{\infty}: \ \|g\|_{E^{\times}}=\sup\limits_{\|f\|_E\leq 1}\int_{\Om} |f\cdot g| d\mu < \infty \}
$$
is the {\it K\"othe  dual space} of $E$, \ then  symmetric space $(E, \| \cdot \|_E)$  has
Fatou property (see, for example, \cite[Vol.II, Ch.I, \S 1b]{lt}).
Since $(\cal L^1 + \cal L^{\infty})^{\times\times} = \cal L^1 + \cal L^{\infty}$  and
$(\cal L^1\cap \cal L^{\infty})^{\times\times} = \cal L^1\cap \cal L^{\infty}$ \cite[Ch.2, \S 6, Theorem 6.4]{bs},
the spaces $(\cal L^1 + \cal L^{\infty}, \|\cdot\|_{\cal L^1 + \cal L^{\infty}})$  and  $(\cal L^1\cap \cal L^{\infty},
\|\cdot\|_{\cal L^1 \cap \cal L^{\infty}})$  possess  Fatou property.

A sequence $ \{f_{n}\} \subset \cal L^0$ is said to converge to $ f \in \cal L^0$ in {\it measure topology} $t_\mu$ if
$f_{n} \chi_E \to f\chi_E$ in measure $\mu$ whenever $\mu(E) < \infty$. It is clear that  $f_{n} \to f$ a.e. implies
$f_{n} \to f$ in $t_\mu$.
Note that in the case $\sigma-$finite measure $\mu$ the algebra $(\cal L^0, t_\mu)$  is a complete  metrizable topological algebra.

In what follows we rely on the fact that any symmetric space with Fatou property is fully symmetric and its unit ball is closed in $t_\mu$ \cite[Ch.IV, \S 3, Lemma 5]{ka}.

Define
$$
\cal R_\mu= \{f \in \cal L^1 + \cal L^{\infty}: \ \mu_t(f) \to 0 \text{ \ as \ } t\to \ii\}.
$$
By \cite[Ch.II, \S 4, Lemma 4.4]{kps}, $(\cal R_\mu, \|\cdot\|_{\cal L^1 + \cal L^\ii})$ is a symmetric space. In addition,
$\cal R_\mu$ is the closure of $\cal L^1\cap \cal L^{\infty}$  in $\cal L^1 + \cal L^{\infty}$ (see \cite[Ch.II, \S 3, Section 1]{kps}).
Furthermore, it follows from definitions of $\cal R_\mu$ and $\| \cdot \|_{\cal L^1 + \cal L^{\infty}}$ that if
$$
f \in \cal R_\mu, \ g \in L^1+L^{\infty} \text{ \ and \ } g \prec\prec f,
$$
then $g\in \cal R_\mu$ and $\| g\|_{\cal L^1 + \cal L^{\infty}} \leq \| f\|_{\cal L^1 + \cal L^{\infty}}$. Therefore
$(\cal R_\mu, \|\cdot\|_{\cal L^1 + \cal L^{\infty}})$ is a   fully symmetric space.
It is clear that if $\mu(\Om) < \infty$, then $\cal R_\mu = \cal L^1$.

\begin{pro}\label{p21}
If $\mu(\Om) = \infty$, then a symmetric space $E \su \cal L^1+\cal L^{\infty}$ is contained in
$\cal R_\mu$ if and only if $\mathbf 1\notin E$.
\end{pro}
\begin{proof}
As  $\mu(\Om) = \infty$, we have $\mu_t(\mathbf 1) = 1$ for all $t > 0$, hence $\mathbf 1 \notin \cal R_\tau$. Therefore
 $ E$ is not contained in $\cal R_\mu$ whenever $\mathbf 1 \in E$.

Let  $\mathbf 1  \notin  E$. If $f \in E$ and $\lim_{t\rightarrow \infty} \mu_t(f) = \alpha >0$, then
$$
\mu_t(\mathbf 1) \equiv 1 \leq \frac1{\alpha} \mu_t(f),
$$
implying $\mathbf 1 \in  E$, a contradiction. Thus $\mathbf 1  \notin  E$ entails $ E \subset  \cal R_\mu$.
\end{proof}

To outline the scope of applications of what follows, we assume that $\mu(\Om) = \infty$ and present some examples of fully symmetric spaces $E$ with $\mathbf 1\notin E$.

1. Let $\Phi$ be an {\it Orlicz function}, that is, $\Phi:[0,\ii)\to [0,\ii)$ is convex, continuous at $0$ and such that $\Phi(0)=0$ and $\Phi(u)>0$ if $u\ne 0$.  Let
$$
L^\Phi=L^\Phi(\Om,\mu)=\left \{ f \in \cal L^0(\Om,\mu): \  \int_{\Om} \left (\Phi\left (\frac {|f|}a \right )\right ) d \mu
<\ii \text { \ for some \ } a>0 \right \}
$$
be the corresponding {\it Orlicz space},  and let
$$
\| f\|_\Phi=\inf \left \{ a>0:  \int_{\Om} \left (\Phi\left (\frac {|f|}a \right )\right ) d \mu \leq 1 \right \}
$$
be the {\it Luxemburg norm} in $L^\Phi$. \ It is well-known that  $(L^\Phi, \| \cdot\|_\Phi)$
is a fully symmetric space  with  Fatou property.
Since  $\mu(\Om) = \infty$, we have $\int_{\Om} \left (\Phi\left (\frac {\mathbf 1}a \right )\right ) d \mu = \ii$ for all
$a>0$, hence $\mathbf 1  \notin   L^\Phi$.

2. A  symmetric space  $(E, \| \cdot \|_E)$ is said to have {\it order continuous norm} if
$$\| f_{n}\|_E\downarrow 0 \ \ \text{whenever} \ \ f_{n}\in E_+ \ \ \text{and} \ \  f_{n}\downarrow 0.
$$
If $E$  is a symmetric space  with order continuous norm, then
$\mu\{|f| > \lambda\}< \infty$  for all $f \in E$ and $\lambda > 0$, so  $E \su \cal R_\mu $; in particular, $\mathbf 1  \notin   E$.

3.  Let $\varphi$ be an increasing  concave function on $[0, \infty)$ with $\varphi(0) = 0$ and
$\varphi(t) > 0$ for some $t > 0$, and let
$$
\Lambda_\varphi=\Lambda_\varphi(\Om,\mu) = \left \{f \in \cal L^0(\Om,\mu): \ \|f \|_{\Lambda_\varphi} =
\int_0^{\infty} \mu_t(f) d \varphi(t) < \infty \right \},
$$
the corresponding {\it Lorentz space}. It is well-known that  $\Lambda_\varphi$
is a fully symmetric space  with  Fatou property; in addition, if $\varphi(\infty) = \ii$, then
$\mathbf 1  \notin  \Lambda_\varphi$.

4. A Banach lattice  $(E,\|\cdot\|_E)$  is called {\it $q$-concave}, $1 \leq q < \infty$, if there exists a
constant  $M>0$ such that
$$
\left (\sum_{i=1}^n\|x_i\|^q\right )^{\frac{1}{q}} \leq M \left \| \left (\sum_{i=1}^n|x_i|^q\right )^{\frac{1}{q}}\right \|_E
$$
for every finite set  $\{x_i\}_{i=1}^n \su E$.
If a Banach lattice  $(E,\|\cdot\|_E)$  is a $q$-concave for some $1 \leq q < \infty$, then  there is no a sublattice of $E$ isomorphic to $l^{\infty}$, and the norm $\| \cdot \|_E$ is order continuous \cite[Corollary 2.4.3]{pm}. Therefore, if a fully symmetric function space $(E, \|\cdot\|_{E})$ is  $q$-concave, then $\mathbf 1\notin E$.

5. Let $(E(0,\ii), \| \cdot \|_{E(0,\ii)})$ be a fully symmetric space. If $s>0$, let the bounded linear operator
$D_s$ in $E(0,\ii)$  be given by $D_s(f)(t) = f(t/s), \ t > 0$.
The  {\it Boyd index}  $q_E$ is defined as
$$
q_E=\lim\limits_{s \to +0}\frac{\log s}{\log \|D_{s}\|}.
$$
It is known that $1\leq q_E\leq \ii$  \cite[Vol.II, Ch.II, \S 2b,  Proposition 2.b.2]{lt}.
Since  $\|D_{s}\| \leq \max\{1,s\}$ \cite[Vol.II, Ch.II, \S 2b]{lt}, $\mathbf 1 \in  E(0,\ii)$ would imply
$D_{s}(\mathbf 1) =\mathbf 1$  and \ $\|D_{s}\| =1$ for all $s \in (0,1)$, hence $q_E = \infty$.
Thus, if  $q_E < \ii$, we have $\mathbf 1\notin E(0,\ii)$.

\vskip 5pt
The next property of the fully symmetric space $\cal R_\mu$ is crucial.

\begin{pro}\label{p22}
For every $f\in \cal R_\mu$ and $\ep>0$ there exist $g_{\ep}\in \cal L^1$  and $h_{\ep}\in \cal L^{\ii}$ such that
$f=g_{\ep}+h_{\ep}$ and $\| h_\ep\|_{\ii}\leq \ep$.
\end{pro}

\begin{proof} If
$$
\Om_{\ep}=\{ |f|>\ep\}, \ \ g_{\ep}=f \cdot \chi_{\Om_{\ep}}, \ \ h_{\ep}=f \cdot \chi_{\Om \sm \Om_{\ep}},
$$
then $\| h_{\ep}\|_{\ii}\leq \ep$. Besides, as $f\in \cal L^1+\cal L^{\ii}$, we have
$$
f=g_{\ep}+h_{\ep}=g+h
$$
for some $g\in \cal L^1$, $h\in \cal L^{\ii}$. Then, since $f\in \cal R_\mu$, we have $\mu(\Om_{\ep})<\ii$, which implies that
$$
g_{\ep}=g\cdot \chi_{\Om_{\ep}}+(h-h_{\ep})\cdot \chi_{\Om_{\ep}}\in \cal L^1.
$$
\end{proof}

\section{Dunford-Schwartz operators and absolute contractions}

\begin{df}
A linear operator $T: \cal L^1 + \cal L^{\ii} \to  \cal L^1 + \cal L^{\ii}$ is called a {\it Dunford-Schwartz operator} if
$$
\| T(f)\|_1\leq \| f\|_1 \ \ \forall \ \ f\in L^1 \text{ \ \ and \ \ } \| T(f)\|_{\ii}\leq \| f\|_\ii \ \ \forall \ f \in \cal L^{\ii}.
$$
\end{df}

In what follows, we will write $T\in DS$ ($T\in DS^+$) to indicate that $T$ is a Dunford-Schwartz operator
(respectively, a positive  Dunford-Schwartz operator). It is clear that
$$
\|T\|_{\cal L^1 + \cal L^{\ii} \to \cal  \cal L^1 + \cal L^{\ii}} \leq 1
$$
for all $T\in DS$ and, in addition, $Tf \prec\prec f$ for all $f \in \cal L^1 + \cal L^{\ii}$ \cite[Ch.II, \S 3, Sec.4]{kps}.
Therefore $T(E) \subset E$ for every fully symmetric space
$E$ and
\begin{equation}\label{eq111}
\| T\|_{E \to  E} \leq 1
\end{equation}
(see \cite[Ch.II, \S 4, Sec.2]{kps}).
In particular,  $T(\cal R_\mu) \subset \cal R_\mu$, and the restriction of $T$ on $\cal  R_\mu$ is a linear contraction (also denoted by $T$).

We say that a linear operator $T: \cal L^1\to \cal L^1$ is an {\it absolute contraction} and write $T\in AC$ if
$$
\| T(f)\|_1\leq \| f\|_1 \ \ \forall \ f\in \cal L^1 \text{ \  and \ }  \| T(f)\|_{\ii}\leq \| f\|_{\ii} \ \ \forall \ f\in  \cal L^1 \cap  \cal L^{\ii}.
$$
If $T\in AC$ is positive, we will write $T\in AC^+$.

\begin{df}
A  complete measure space $(\Om, \cal A,\mu)$ is called {\it semifinite} if every subset of $\Om$ of non-zero  measure admits a subset of
finite non-zero  measure. A semifinite measure space $(\Om, \cal A, \mu)$ is said to have the {\it direct sum property} if the Boolean
algebra $(\cal A /\sim)$ of equivalence classes of measurable sets is complete, that is, every subset of $(\cal A /\sim)$ has a least upper bound.
\end{df}

Note that every $\sigma-$finite measure space has the direct sum property. A detailed account on measures with
direct sum property is found in \cite{di}; see also \cite{ku}.

Absolute contractions (or $\cal L^1-\cal L^{\ii}-$contractions)  were considered in \cite{ga} and also in \cite{kr}.
It is clear that  if $T \in DS$, then $T| \cal L^1\in AC$.
It turns out that if  $\cal L^{\ii}, \cal L^1\su \cal L^0_h$
and $(\Om,\mu)$ has  the direct sum property, then $T\in AC$ can be uniquely extended
to a  $\sigma(\cal L^{\ii}, \cal L^1)-$continuous $DS$ operator:

\begin{teo}\label{t31}
Let $(\Om,\mu)$ have the direct
sum property  and let $\cal L^{\ii}, \cal L^1\su \cal L_h^0$. Then for any $T\in AC$ there exists a unique $S \in DS$ such that $S| \cal L^1 = T$ \ and $S| \cal L^{\ii}$ is $\sigma(\cal L^{\ii}, \cal L^1)-$continuous.
\end{teo}

\begin{rem}
In what follows, we will only need Theorem \ref{t31} in the case of $\sigma-$finite measure.
However, since this theorem, and Lemma \ref{l31}
below, are interesting results by themselves, we prove them in more general settings, namely, for a measure with the direct
sum property and for a semifinite measure, respectively.
\end{rem}

In order to prove Theorem \ref{t31}, we will need Lemma \ref{l31} below.  Let us denote
$$
\cal F=\{F\su \Om: \ 0<\mu(F)<\ii\},
$$
and let $\{ F_{\al}\}$ be the directed set consisting of the elements of $\cal F$, ordered by inclusion.

\begin{lm}\label{l31}  If $(\Om,\mu)$ is semifinite, then
$\int \chi_{F_{\al}}g\to \int g$ for all $g\in \cal L^1$.
\end{lm}
\begin{proof}
It is enough to prove the statement for $g\ge 0$. Then $\int \chi_{F_{\al}}g\leq \int g$ for every $\al$ and, since
$\left \{ \int \chi_{F_{\al}}g \right \}$ is an increasing net  of real numbers, we have
$$
\lim_{\al} \int \chi_{F_{\al}}g=\sup_{\al}\int \chi_{F_{\al}}g=s<\ii.
$$
There exists a sequence $\{ f_n\}\su \cal F$ such that
$$
\lim_{n\to \ii} \int \chi_{F_n}g=s.
$$
Without loss of generality, we can assume that $F_n\su F_{n+1}$ for all $n \geq 1$.
Let $E=\cup F_n$ and assume that $g\chi_{\Om \sm E}\neq 0$. Then there is $F\in \cal F$
such that $F\su \Om \sm E$ and $g\chi_F>0$, and we obtain
$$
\lim_{n\to \ii}\int \chi_{F_n\cup F}g=\int \chi_{E\cup F}g>\int \chi_Eg=\lim_{n\to \ii}\int \chi_{F_n}g=s.
$$
This entails that $\sup_{\al}\int \chi_{F_{\al}}g>s$, a contradiction, so
$g\chi_{\Om \sm E}=0$. Consequently,
$$
s=\lim_{n\to \ii} \int \chi_{F_n}g=\int \chi_Eg=\int g,
$$
hence $\int \chi_{F_{\al}}g\to \int g$.
\end{proof}

Let us first establish Theorem \ref{t31} for positive operators.

\begin{teo}\label{t32}
If $(\Om,\mu)$, $\cal L^{\ii}$, and $\cal L^1$ are as in Theorem \ref{t31}, then, given $T\in AC^+$, there exists a unique $S \in DS^+$ such that $S| \cal L^1 = T$ and $S|\cal L^{\ii}$ is $\sigma(\cal L^{\ii}, \cal L^1)-$continuous.
\end{teo}
\begin{proof}
 Since $(\Om,\mu)$ has the direct sum property, Radon-Nikodym theorem is valid \cite[$\S$10, Sec.8]{di},
hence $ (\cal L^1)^* = \cal L^{\ii}$, and the adjoint operator $T^*: \cal L^{\ii} \to \cal L^{\ii}$ is defined.
Besides, $\|T^*\|_{\cal L^{\ii}\to \cal L^{\ii}} = \|T\|_{\cal L^1 \to \cal L^1} \leq 1$ and $T^*$ is also
$\sigma(\cal L^{\ii}, \cal L^1)-$continuous. Moreover,

$$
\int T^*(f)g=\int fT(g)
$$
for all $f\in \cal L^{\ii}$, $g\in \cal L^1$. In particular, it follows that the linear operator $T^*$ is positive.

Now, if $f\in \cal L^{\ii}_+\cap \cal L^1$, then, with $\chi_{F_{\al}}$ as in Lemma \ref{l31}, we have
$$
0 \leq \int T^*(f)\chi_{F_{\al}} = \int f T(\chi_{F_{\al}}) \leq \int f,
$$
for all $\alpha$, hence $T^*(f) \in \cal L^{\ii}_+\cap \cal L^1$.
In addition,
$$
\begin{aligned}
\|T^*(f) \|_1
& = \sup \left \{\left | \int T^*(f)g \right |:\  g \in \cal L^{\ii} =  (\cal L^1)^*,\  \|g\|_{\ii}\leq 1 \right \} \\
&\leq \sup \left \{ \int T^*(f)|g|:\  g \in \cal L^{\ii},\  \|g\|_{\ii}\leq 1 \right \} \\
&=\sup \left \{ \int T^*(f) g:\  g \in \cal L^{\ii}_+,\  \|g\|_{\ii}\leq 1 \right \}.
\end{aligned}
$$
Since, by Lemma \ref{l31}, for every $g \in \cal L^{\ii}_+$ with $\|g\|_{\ii}\leq 1$ we have
$$
\begin{aligned}
\int T^*(f) g
&= \lim_{\al} \int T^*(f) \chi_{F_{\al}}g \\
& \leq \sup \left \{ \int T^*(f)h=\int f T(h):\  h \in \cal L^{\ii}_+\cap \cal L^1,\  \| h\|_{\ii} \leq 1 \right \} \leq \int f,
\end{aligned}
$$
it follows that $\| T^*(f)\|_1\leq \| f\|_1$ whenever $f\in \cal L^{\ii}_+\cap \cal L^1$.  Therefore $T^*$
is a positive linear $\| \cdot \|_1-$contraction on  $\cal L^{\ii}\cap \cal L^1$  and,
since $\cal L^{\ii}\cap \cal L^1$ is dense in $\cal L^1$, it uniquely extends to a  $DS^+$ operator, which we also denote by $T^*$.

For the $\sigma(\cal L^{\ii}, \cal L^1)-$continuous adjoint operator $T^{**}:  (\cal L^{\ii})^{\ast} \to (\cal L^{\ii})^{\ast}$  and  all \\ $f, \ g\in \cal L^{\ii}\cap \cal L^1$, we have
$$
\int T^{**}(f)g=\int fT^*(g) = \int T(f) g,
$$
hence $ T^{**}(f) = T(f)$ whenever $f \in \cal L^{\ii}\cap \cal L^1$. In the same way as $T^*$,
$T^{**}$  uniquely extends to a $DS^+$ operator (which we also denote by $T^{**}$)
such that  $T^{**}(g) = T(g)$ for every $g \in \cal L^1$ and $T^{**}|\cal L^{\ii}$ is $\sigma(\cal L^{\ii}, \cal L^1)-$continuous.

Let $S\in DS^+$ be another operator such that $S(g) = T(g)$ for every $g \in \cal L^1$ and  $S|\cal L^{\ii}$ is $\sigma(\cal L^{\ii}, \cal L^1)-$continuous. Given $f \in \cal L^{\ii}$ and $g \in \cal L^1$, it follows from Lemma \ref{l31} that $ \int g \chi_{F_{\al}}f \to \int gf$, that is, $\chi_{F_{\al}}f \rightarrow f$ in $\sigma(\cal L^{\ii}, \cal L^1)-$topology, so $\cal L^{\ii}\cap \cal L^1$ is $\sigma(\cal L^{\ii}, \cal L^1)-$dense in $\cal L^{\ii}$. Therefore $S|\cal L^{\ii} = T^{**}|\cal L^{\ii}$,
which completes the proof.
\end{proof}

We shall recall now the following statement on the existence and properties of linear modulus of a bounded linear operator $T:\cal L^1\to \cal L^1$ ($T:\cal L^{\ii} \to \cal L^{\ii}$) (see \cite[Ch.4, \S 4.1, Theorem 1.1]{kr}):

\begin{teo}\label{t33}
Let $\cal L^{\ii}, \cal L^1\su \cal L_h^0$. Then for any bounded linear operator \\ $T:\cal L^1\to \cal L^1$ ($T:\cal L^{\ii} \to \cal L^{\ii}$) there exists a unique positive
bounded linear operator $|T|:\cal L^1\to \cal L^1$  (respectively, $|T|:\cal L^{\ii} \to \cal L^{\ii}$) such that
\begin{enumerate} [(i)]
\item $ \| \ |T| \ \| = \| T\|$;

\item $|T(f)|\leq |T|(|f|)$ for all $f\in \cal L^1$ (respectively, for all $f\in \cal L^{\ii}$);

\item $|T|(f )=\sup\{ |T(g)|: g\in \cal L^1, |g|\leq f\}$ for all $f\in \cal L^1_+$
\\ (respectively, $|T|(f)=\sup\{ |T(g)|: g \in \cal L^{\ii}, |g|\leq f\}$ for all $f \in \cal L^{\ii}_+$).

\end{enumerate}

\end{teo}

The operator $|T|$ is called the {\it linear modulus} of $T$.
In addition, $|T|$ satisfies the following properties \cite[Ch.4, \S 4.1, Theorem 1.3, Proposition 1.2 (d),(e)]{kr}).

\begin{pro}\label{p32}
If $T \in AC$, then
\begin{enumerate} [(i)]
\item  $|T^k(f)|\leq |T|^k(|f|)$ for every $f\in \cal L^1, \ k=1,2, \dots$;
\item $\|T\|_{1,\infty}:= \sup\{\|Tf\|_{\infty}: f \in \cal L^1 \cap \cal L^{\ii}, \ \|f\|_{\infty} \leq 1 \} = \| \ |T| \ \|_{1,\infty}$;
\item $ |T^*| = |T|^*$ (in the case direct sum property of $(\Om,\cal A, \mu)$).
\end{enumerate}
\end{pro}

Here is a proof of Theorem \ref{t31}.

\begin{proof} By virtue of Theorem \ref{t33} (i) and Proposition \ref{p32} (ii), $|T| \in AC^+$.
The adjoint operators $T^*$ and $|T|^*$ are contractions in $\cal L^{\ii}$ such that $|T^*| = |T|^*$,
by Proposition \ref{p32} (iii).
As in proof of Theorem \ref{t32}, $|T|^*$ uniquely extends to a $DS^+$ operator, which we also denote by $|T|^*$.
It follows from
$$
\|T^* f\|_1 \leq  \| |T^*| (|f|)\|_1 =  \| |T|^* (|f|)\|_1 \leq \| f\|_1,  \ f \in \cal L^{\ii} \cap \cal L^1,
$$
that $T^*$ is a $\|\cdot\|_1-$contraction in $\cal L^{\ii} \cap \cal L^1$. Therefore, as in the proof of Theorem \ref{t32},
$T^*$ and then $T^{**}$  can be uniquely extended to $DS$ operators.

Repeating the argument at the end of proof of Theorem \ref{t32}, we obtain the result.
\end{proof}
An important consequence of Theorem \ref{t31} is that  if $\cal L^{\ii}, \cal L^1\su \cal L_h^0$, then, given \\ $T\in AC$, one can assume, without loss of generality, that $T\in DS$, so that $\| T(f)\|_{\ii}\leq \| f\|_{\ii}$ for every $f\in \cal L^{\ii}$.

\begin{rem}\label{r10}
Let $E$ be a subspace of $\cal L^1+\cal L^\ii$ such that $T(E)\su E$ if $T\in DS$.
In what follows, if a.e. convergence of conventional or weighted ergodic averages holds for $T\in DS$ and every
$f\in E_h$, these averages also converge a.e. whenever $f\in E$ and $T\in DS^+$
because then $T(E_h)\su E_h$. Therefore,
we routinely assume that $T\in DS$ if $\cal L^{\ii},\cal L^1\su \cal L_h^0$ and $T\in DS^+$ in the general case.
\end{rem}

\section{Dunford-Schwartz pointwise ergodic theorem in fully symmetric spaces}

Dunford-Schwartz theorem on pointwise convergence of the ergodic averages
\begin{equation}\label{eq01}
a_n(f)=\frac 1 n \sum_{k=0}^{n-1}T^k(f)
\end{equation}
for $T\in DS$ acting in the $L^p-$space, $1\leq p<\ii$, of real valued functions of an
arbitrary measure space was established in \cite{DS}; see also \cite[Theorem VIII.6.6]{ds}.
Note that, since the set on which  the function $T^k(f)\in \cal L^p$, $1\leq p<\ii$,
$k=0,1,2,\dots$, does not equal to zero
is a countable union of sets of finite measure, one can assume that the measure is $\sigma-$finite.

In this section we prove Dunford-Schwartz pointwise
ergodic theorem in $\cal R_\mu$, Theorem \ref{t3}, for real and complex valued (when $T\in DS^+$) functions,
arguably the most general version of the classical result.

In view of Remark \ref{r10}, Dunford-Schwartz pointwise ergodic theorem in $\cal L^1$ can be stated as follows.

\begin{teo}\label{t2}
Let $(\Om, \mu)$  be an arbitrary measure space. Assume that either $\cal L^{\ii}, \cal L^1\su \cal L_h^0$ and
$T\in DS$ or $\cal L^{\ii}, \cal L^1\su  \cal L^0$ and $T\in DS^+$. Then the averages (\ref{eq01})
converge a.e. for all $f\in \cal L^1$.
\end{teo}

Below is an extension of Theorem \ref{t2} to the fully symmetric space  $\cal R_\mu$.

\begin{teo}\label{t3} Let $(\Om,\mu)$  be a  measure space. Assume that either $\cal L^{\ii}, \cal L^1\su \cal L_h^0$ and
$T\in DS$ or $\cal L^{\ii}, \cal L^1\su  \cal L^0$ and $T\in DS^+$. If $f \in \cal R_\mu$, then the averages (\ref{eq01})
converge a.e. to some $\widehat f \in  \cal R_\mu$.
\end{teo}

\begin{proof}
Without loss of generality, assume that $f\geq 0$.
By Proposition \ref{p22}, given $k=1,2, \dots$, there are $0 \leq g_k\in \cal L^1$ and $0 \leq h_k\in \cal L^{\ii}$ such that
$$
f=g_k+h_k \text{ \ and \ } \| h_k\|_{\ii}\leq \frac 1k.
$$
Since $\{ g_k\} \su \cal L^1$  it follows from Theorem \ref{t2}  that the averages (\ref{eq01}) converge a.e. for each $g_k $. As
 $\{ a_n(g_k)\} \subset \cal L_h^1$, we can assume without loss of generality that
$$
\limsup_n a_n(g_k)(\om)=\liminf_n a_n(g_k)(\om)
$$
for all $\om \in \Om$ and every $k$. Then, for a fixed $k$ and $\om \in \Om$, we have
$$
0 \leq \Delta(\om)=\limsup_n a_n(f)(\om)-\liminf_n a_n(f)(\om)=
$$
$$
=\limsup_n a_n(h_k)(\om)-\liminf_n  a_n(h_k)(\om)\leq 2 \sup_n | a_n(h_k)(\om)|,
$$
which together with $\| a_n(h_k) \|_{\ii}\leq \frac 1k$, $n=1,2,\dots$,  implies that there exists $\Om_k \in \Om$ with $\mu(\Om\sm\Om_k)=0$
such that $0 \leq \Delta(\om)\leq \frac 2k$ whenever $\om \in \Om_k$. Then, letting
$\Om_f=\cap_k\Om_k$, we obtain $\mu(\Om\sm\Om_f)=0$ and
$\Delta(\om)\leq \frac 2k$ for all $\om\in\Om_f$ and every $k$.
Therefore $\Delta(\om)=0$, hence
$$
\limsup_n a_n(f)(\om)=\liminf_n a_n(f)(\om), \ \ \om\in\Om_f,
$$
and we conclude that the sequence $\{ a_n(f)\}$ converges a.e. to a $\mu-$measurable function $\widehat f$ on $\Om$.
Note that, since $\cal L^0$ is complete in the measure topology, $\widehat f$ cannot be infinite on a set of positive measure,
hence $\widehat f\in \cal L^0$.

Since  $\cal L^1+\cal L^{\infty}$ satisfies Fatou property, its unit ball is closed in measure topology $t_\mu$, and (\ref{eq111}) implies that $\widehat f \in \cal L^1+\cal L^{\infty}$.

Since $ a_n(f)\to \widehat f$ in $t_\mu$, it follows that  $$
\mu_t( a_n(f))\to \mu_t(\widehat f) \ \ \text{a.e.  on} \ \ (0,\ii)
$$
(see, for example, \cite[Ch.II, \S 2, Property $11^\circ$]{kps}).
Besides, the inclusion $T \in DS$ implies  that    $\mu_t( a_n(f))\prec\prec \mu_t( f)$ for everyl $n$ (see, for example,
\cite[Ch.II, \S 3, Sec.4]{kps}). Utilizing Fatou property for $\cal L^1(0,s)$ and the measure convergence
$$
\mu_t( a_n(f))\to \mu_t(\widehat f) \ \ \text{on} \ \  (0,s),
$$
we derive
$$
\int \limits_0^s\mu_t(\widehat f)dt \leq \sup_{n\geq1} \int \limits_0^s\mu_t(a_n(f))dt \leq \int\limits_0^s\mu_t(\widehat f)dt
$$
for all $s>0$, that is, $\mu_t(\widehat f)\prec\prec \mu_t( f)$.
Since $\cal R_\mu$ is a fully symmetric space and $f\in \cal R_\mu$, it follows that $\widehat f\in \cal R_\mu$.
\end{proof}

Now we present a version of Theorem \ref{t3} for a fully symmetric space $E \subset  \cal R_\mu$.

\begin{teo}\label{t5} Let $(\Om,\mu)$  be an infinite measure space, and let $E$ be a fully symmetric space with
$\mathbf 1\notin E$. Assume that either $E \su \cal L_h^0$ and $T\in DS$ or $E \su \cal L^0$ and $T\in DS^+$.
Then for every \ $f \in E$ the averages (\ref{eq01})
converge a.e. to some \ $\widehat f \in  E$.
\end{teo}
\begin{proof}
By Proposition \ref{p21}, $E \subset  \cal R_\mu$. Then, by Theorem \ref{t3}, given $f \in E$, the averages (\ref{eq01})
converge a.e. to some $\widehat f \in \cal R_\mu$. Since $E$ is a fully symmetric space, it follows as in Theorem \ref{t3} that
$\widehat f \in   E$.
\end{proof}

The next theorem implies that, in the model case $\Om=(0,\ii)$, if
a symmetric space $E\su \cal L^1+\cal L^{\infty}$ is such that $E\sm \cal R_\mu\neq \emptyset$,
then Dunford-Schwartz pointwise ergodic theorem does nod hold in $E$.

\begin{teo}\label{t4}
Let $\Om =(0, \infty)$, and let $\mu$ be Lebesgue measure. Then, given $f \in (\cal L^1+\cal L^{\infty}) \setminus \cal R_\mu$, there exists $T\in DS$ such that the averages $a_n(\mu_t(f))$ do not converge a.e.
\end{teo}
\begin{proof}
Since  $f \in (\cal L^1+\cal L^{\infty}) \setminus \cal R_\mu$, it follows that $\mu_t(f) \geq \ep$ for all $t >0$ and some $\varepsilon >0$. Without loss of generality we can assume that $\mu_t(f) \geq 1$ for all $t >0$.

Let $\{n_k\}_{k=1}^{\infty}$ be an increasing sequence of integers with $n_0=0$ (the choice of this sequence is given below).  Consider the function
$$
\varphi(t)=\sum_{k=0}^{\infty}\left (\chi_{[n_k,n_{k +1}-1)}(t) - \chi_{[n_{k +1}-1,n_{k+1})}(t)\right )
$$
and the operator $T$ in $\cal L^1+\cal L^{\ii}$ defined by
$$
T(f)(t) = \varphi(t) f(t+1), \ \ f \in \cal L^1+\cal L^{\ii}.
$$
It is clear that $T \in DS$  and
$$
a_n(\mu_t(f))=
=\frac 1 n \left (\mu_{t}(f)+\sum_{k=1}^{n-1}\varphi(t)\varphi(t+1)\cdot \dotsc \cdot\varphi(t+k-1)\mu_{t+k}(f)\right ).
$$
Show that the averages $a_n(\mu_t(f))$ do not converge a.e. Fix $\ep\in (0,1)$. If $t \in (\ep, 1)$, then
$$
a_{n_1}(\mu_t(f)) =\frac 1 {n_1}\sum_{m=0}^{n_1-1}\mu_{t+m}(f)\ge 1.
$$
Next, since
$$
a_{n_2}(\mu_t(f)) =\frac 1 {n_2}\left (\sum_{m=0}^{n_1-1}\mu_{t+m}(f)-\sum_{m=n_1}^{n_2-1}\mu_{t+m}(f)\right )\leq
\frac 1{n_2}\left (n_1\mu_\ep(f)-(n_2-n_1)\right ),
$$
one can choose big enough $n_2 $ so that $a_{n_2}(\mu_t(f))<-\frac 12$ for every $t\in (\ep,1)$. As
$$
a_{n_3}(\mu_t(f)) =\frac 1 {n_3}\left (\sum_{m=0}^{n_1}\mu_{t+m}(f)-\sum_{m=n_1+1}^{n_2-1}\mu_{t+m}(f)
+\sum_{m=n_2+1}^{n_3-1}\mu_{t+m}(f)\right ),
$$
there exists $n_3$ such that $a_{n_3}(\mu_t(f))>\frac 12$ for every $t\in (\ep,1)$.

Continuing this process, we construct a sequence of positive integers $n_1 < n_2 < ...$ such that for every $t \in (\ep, 1)$,
$$
a_{n_{2k-1}}(\mu_t(f)) >\frac 12  \text{ \ \ and \ \ }
a_{n_{2k}}(\mu_t(f)) <- \frac 1 2, \ \ k=1,2,\dots
$$
Thus, the averages $a_n(\mu_t(f))$ diverge for each $t \in (\ep, 1)$.
\end{proof}

\section{Weigthed ergodic theorem for Dunford-Schwartz operators   in fully symmetric spaces}

Let $(\Om,\mu)$  be a measure space. Given $T\in DS$, a bounded sequence
$\ol \bt =\{ \bt_k\}_{k=0}^{\ii}$ of complex numbers is called a {\it good weight} for $T$
if the sequence of weigthed ergodic averages
\begin{equation}\label{eq41}
a_n(\ol \bt, f)=\frac 1n \sum_{k=0}^{n-1} \beta_k T^k(f)
\end{equation}
converges $\mu-$a.e. for every $f \in \cal L^1 $ (see, for example, \cite{clo}).

 In \cite{bo,bl,clo,lot}, various classes of good weights for $T\in DS$ in $L^p-$spaces of real valued functions on finite and infinite measure spaces were studied. In particular, it follows from a general measure space extension
 of \cite[Theorem 2.19]{bo} given  in \cite{clo}  that bounded Besicovitch sequences are good weights
for any $L^1-$contraction with mean ergodic modulus,  in particular, for Dunford-Schwartz operators,
in an arbitrary measure space. The convergence also holds when
the functions are not assumed to be real valued but $T\in DS^+$; see Theorem \ref{t4} below.

Let $\Bbb C_1$ be the unit circle in the field $\mathbb C$  of complex numbers, and let $\mathbb Z$ be the set of all integers. A function $P : \mathbb Z \to \mathbb C$ is said to be a {\it trigonometric polynomial} if
$P(k)=\sum_{j=1}^{s} z_j\lb_j^k$, $k\in \mathbb Z$, for some $s\in \mathbb N$, $\{ z_j \}_1^s \subset \mathbb C$, and $\{ \lb_j \}_1^s \subset \mathbb C_1$.
A sequence $\{ \beta_k \} \subset \Bbb C$ is called a {\it bounded Besicovitch sequence} if

(i) $| \beta_k | \leq C < \ii$ for all $k$;

(ii) for every $\ep >0$ there exists a trigonometric polynomial $P$ such that
\begin{equation}\label{eq041}
\limsup_n \frac 1n \sum_{k=0}^{n-1} | \beta_k - P(k) | < \ep .
\end{equation}

From what was just noticed and taking into account Remark \ref{r10} we have the following.
\begin{teo}\label{t4} Assume that either $T\in DS$ and $\cal L^1\su \cal L^0_h$ or $T\in DS^+$ in $\cal L^1\su \cal L^0$.
Then any bounded Besicovitch sequence $\ol \bt=\{ \beta_k \}$ is a good weight for $T$.
\end{teo}

Now we will show that if $\ol \bt$ is a good weight for
$T\in DS$, then the averages (\ref{eq41}) converge a.e. in any fully symmetric space $E$  with $\mathbf 1\notin E$.

It should be pointed out that the boundedness of a sequence $\ol \bt$ implies that
\begin{equation}\label{eq42}
a_n(\ol \bt,f)(E) \subset E \text{ \ \ and \ \ } \|a_n(\ol \bt,f)\|_{E \to  E} \leq \sup_{k\ge 0} | \bt_k| < \infty.
\end{equation}
for every fully symmetric space $E$ and $T\in DS$.

As before, we begin with the fully symmetric space  $\cal R_\mu$.

\begin{teo}\label{t41} Let $(\Om,\mu)$  be a measure space. Let  $\ol \bt =\{ \bt_k\}_{k=0}^{\ii}$ be a  good weight for $T\in DS$. If $f \in \cal R_\mu$, then the averages (\ref{eq41}) converge a.e. to some $\widehat f \in  \cal R_\mu$.
\end{teo}

\begin{proof}
Without loss of generality assume that  $(\Om,\mu)$ is $\sigma-$finite and $f\geq 0$.
By Proposition \ref{p22}, given $k=1,2, \dots$, there are $g_k\in \cal L^1$ and $h_k\in \cal L^{\ii}$ such that
$$
f=g_k+h_k \text{ \ and \ } \| h_k\|_{\ii}\leq \frac 1k.
$$
Since $\{ g_k\} \su \cal L^1$  and \ ${\bf \beta} =\{ \bt_k\}_{k=0}^{\ii}$ \ is a  good weight for $T$   the averages (\ref{eq41}) converge a.e. for all \ $g_k $. Thus we can assume without loss of generality that
$$
\limsup_n a_n(\ol \bt,g_k)(\om)=\liminf_n a_n(\ol \bt,g_k)(\om)
$$
for all $\om \in \Om$ and every $k$. Then, for a fixed $k$ and $\om \in \Om$, we have
$$
0 \leq \Delta(\om)=\limsup_n a_n(\ol \bt,f)(\om)-\liminf_n a_n(\ol \bt,f)(\om)=
$$
$$
=\limsup_n a_n(\ol \bt,h_k)(\om)-\liminf_n a_n(\ol \bt,h_k)(\om)\leq 2 \sup_n | a_n(\ol \bt,h_k)(\om)|,
$$
which together with $\| a_n(\ol \bt,h_k) \|_{\ii}\leq \frac 1k$, $n=1,2,\dots$,  implies that there exists $\Om_k \in \cal A$ with $\mu(\Om\sm\Om_k)=0$
such that $0 \leq \Delta(\om)\leq \frac 2k$ whenever $\om \in \Om_k$. Then, letting
$\Om_f=\cap_k\Om_k$, we obtain $\mu(\Om\sm\Om_f)=0$ and
$\Delta(\om)\leq \frac 2k$ for all $\om\in\Om_f$ and every $k$.
Therefore $\Delta(\om)=0$, hence
$$
\limsup_n a_n(\ol \bt,f)(\om)=\liminf_n a_n(\ol \bt,f)(\om), \ \ \om\in\Om_f,
$$
and we conclude that the sequence $\{ a_n(\ol \bt,f)\}$ converges a.e. to a $\mu-$measurable function $\widehat f$ on $\Om$.
Note that, since $\cal L^0$ is complete in the measure topology $t_\mu$, the function $\widehat f$ cannot be infinite on a set of positive measure, that is, $\widehat f\in \cal L^0$.

Since  $\cal L^1+\cal L^{\infty}$ satisfies Fatou property, its unit ball is closed in measure, so (\ref{eq42}) implies that $\widehat f \in \cal L^1+\cal L^{\infty}$.

As $a_n(\ol \bt, f)\to \widehat f$  in $t_\mu$, it follows that
$$
\mu_t(a_n(\ol \bt, f))\to \mu_t(\widehat f) \ \ \text{a.e.  on} \ \ (0,\ii)
$$
(see, for example, \cite[Ch.II, \S 2, Property $11^\circ$]{kps}).

Setting $M(\ol \bt)= \max\{1,\sup_{k\geq1} | \bt_k|\}$, we have $\frac1{M(\ol \bt)}(\frac1{n} \sum_{k=0}^{n-1} \beta_k T^k) \in DS$, hence  $\mu_t(\frac1{M(\ol \bt)}a_n(\ol \bt,f))\prec\prec \mu_t( f)$
(see, for example, \cite[Ch.II, \S 3, Section 4]{kps}). Using
Fatou property for $\cal L^1(0,s)$ and the measure convergence
$$
\mu_t\left (\frac1{M(\ol \bt)}a_n(\ol \bt,f)\right )\to \mu_t\left (\frac{\widehat f}{M(\ol \bt)}\right )
$$
on $(0,s)$, we derive
$$
\int \limits_0^s\mu_t\left (\frac{\widehat f}{M(\ol \bt)}\right )dt\leq \sup_{n\geq1}
\int \limits_0^s\mu_t\left (\frac1{M(\ol \bt)}a_n(\ol \bt,f)\right )dt \leq \int\limits_0^s\mu_t\left (\frac{f}{M(\ol \bt)}\right )dt
$$
for all $s>0$, that is, $\mu_t(\widehat f)\prec\prec \mu_t( f)$.
Since $\cal R_\mu$ is a fully symmetric space and $f\in  \cal R_\mu$, it follows that $\widehat f\in \cal R_\mu$.
\end{proof}

Since, by Proposition \ref{p21}, $E \subset  \cal R_\mu$ for every symmetric space $E$ with  $\mathbf 1\notin E$, utilizing Theorem \ref{t41} and repeating the ending of its proof, we obtain the following.

\begin{teo}\label{t42} Let $(\Om,\mu)$ be an infinite measure space, and let $T$ and $\ol \bt$ be as in Theorem \ref{t41}.
Assume that $E$ is a fully symmetric space with  $\mathbf 1\notin E$.  If $f \in E$, then the averages (\ref{eq41}) converge a.e. to some \ $\widehat f \in  E$.
\end{teo}

In view of Theorems  \ref{t42} and  \ref{t4}, we now have the following.

\begin{cor}\label{c41} Let $(\Om,\mu)$ be an infinite measure space,
and let $E$ be a fully symmetric space  with $\mathbf 1\notin E$.
Assume that either $E\su \cal L^0_h$ and $T\in DS$ or $E\su \cal L^0$ and $T\in DS^+$.  Then
for every $f\in  E$ and a bounded Besicovitch  sequence $\{ \beta_k \}$ the averages (\ref{eq41}) converge a.e. to some
 $\widehat f \in E$.
\end{cor}

\section{Wiener-Wintner-type ergpdic theorems in fully symmetric spaces}

Assume that $(Y,\nu)$ is a finite measure space and $\phi:Y\to Y$ a measure preserving transformation (m.p.t.).
If $(\Om, \mu)$ is a measure space, $\tau:\Om \to \Om$ a m.p.t., $f:\Om \to \mathbb C$, and $g\in L^1(Y)$, denote
\begin{equation}\label{eq51}
a_n(f,g)(\omega,y)=\frac 1n\sum_{k=0}^{n-1}f(\tau^k\omega)g(\phi^ky)
\end{equation}
Here is an extension of Bourgain's Return Times theorem to $\sigma-$finite measure \cite[p.101]{as0}.

\begin{teo}\label{t51}
Let $(\Om,\mu)$ be a $\sigma-$finite measure space, $\tau: \Om \to \Om$ be a m.p.t., and let $A\su \Om$, $\mu(A)<\ii$.
Then there exists a set $\Om_A \su \Om$ such that $\mu(\Om \sm \Om_A)=0$ and for any $(Y,\nu,\phi)$ and $g\in \cal L^1(Y)$
the averages
$$
a_n(\chi_A,g)(\omega,y)=\frac 1n\sum_{k=0}^{n-1}\chi_A(\tau^k\omega)g(\phi^ky)
$$
converge $\nu-$a.e. for all $\omega\in \Om_A$.
\end{teo}
The next theorem is  a version of Theorem \ref {t51} where the functions $\chi_A$ and $g\in \cal L^1(Y)$
are replaced by $f\in \cal L^1(\Om)$ and  $g\in \cal L^\ii(Y)$, respectively.

\begin{teo}\label{t52}
Let $(\Om,\mu)$ be a measure space, $\tau: \Om \to \Om$ a m.p.t., and let $f\in \cal L^1(\Om)$.
Then there exists a set $\Om_f\su \Om$ with $\mu(\Om\sm \Om_f)=0$ such that for any $(Y,\nu,\phi)$
and $g\in \cal L^\ii(Y)$ the averages (\ref{eq51}) converge $\nu-$a.e. for all $\omega\in \Om_f$.
\end{teo}
\begin{proof} Assume first that $\mu$ is $\sigma-$finite. Fix $f\in \cal L^1(\Om)$. Then there exist
$\{ \lb_{m,i}\} \su \mathbb C$ and
$A_{m,i}\su \Om$ with $\mu(A_{m,i})<\ii$, $m=1,2,\dots$, $1\leq i\leq l_m$, such that
$$
\| f-f_m\|_1\to 0, \text{ \ where \ } f_m=\sum_{i=1}^{l_m}\lb_{m,i}\chi_{A_{m,i}}.
$$
If
$$
\Om_{m,j}=\left \{\omega \in \Om: \sup_n\frac 1n \sum_{k=0}^{n-1}|f-f_m|(\tau^k(\omega)) >\frac 1j \right \},
$$
then, due to the maximal ergodic  inequality, we have
$$
\mu(\Om_{m,j})\leq j\|f-f_m\|_1,
$$
which implies that $\mu (\cap_m \Om_{m,j})=0$ for a fixed $j$. Therefore, denoting
$$
\Om_0=\Om \sm \bigcup_j\bigcap_m \Om_{m,j},
$$
we obtain $\mu(\Om\sm \Om_0)=0$.

If $\omega \in \Om_0$, then $\omega \notin \Om_{m_j,j}$ for every $j$ and some $m_j$, and therefore
\begin{equation}\label{eq52}
\sup_n\frac 1n \sum_{k=0}^{n-1}|f-f_{m_j}|(\tau^k(\omega))\leq \frac 1j \text{ \ for all \ } j \text{ \ and \ } \omega\in \Om_0.
\end{equation}

Now, by Theorem \ref{t51}, there exist $\Om_{j,i}\su \Om$ with $\mu(\Om \sm \Om_{j,i})=0$ such that for every $(Y,\nu,\phi)$ and
$g\in \cal L^\ii(Y)$ the averages
$$
\frac 1n\sum_{k=0}^{n-1}\chi_{A_{m_j,i}}(\tau^k(\omega))g(\phi^ky)
$$
converge $\nu-$a.e. for all $\omega\in \Om_{j,i}$. Then, letting
$$
\Om_f=\bigcap\limits_{j=1}^\ii\bigcap\limits_{i=1}^{l_{m_j}}\Om_{j,i}\bigcap \Om_0,
$$
we obtain $\mu(\Om \sm \Om_f)=0$.

If we pick any $(Y,\nu,\phi)$ and $g\in \cal L^\ii(Y)$, then the averages $a_n(f_{m_j},g)(\omega,y)$
converge $\nu-$a.e. for every $j$ and all $\omega\in \Om_f$, and it follows that there are
$Y_0\su Y$ with $\nu(Y\sm Y_0)=0$ and $C>0$
such that $|g(\phi^ky)|\leq C$ for all $k$ and $y\in Y_0$ and
$$
\liminf_n\operatorname{Re}a_n(f_{m_j},g)(\omega,y)=\limsup_n\operatorname{Re}a_n(f_{m_j},g)(\omega,y),
$$
$$
\liminf_n\operatorname{Im}a_n(f_{m_j},g)(\omega,y)=\limsup_n\operatorname{Im}a_n(f_{m_j},g)(\omega,y),
$$
for all $y\in Y_0$, $k$, and $\omega\in \Om_f$.

Let $\omega\in \Om_f$ and $y\in Y_0$. Given $k$, taking into account (\ref{eq52}), we have
$$
\Delta(\omega,y)=\limsup_n\operatorname{Re}a_n(f,g)(\omega,y)-\liminf_n\operatorname{Re}a_n(f,g)(\omega,y)=
$$
$$
\limsup_n\operatorname{Re}a_n(f-f_{m_j},g)(\omega,y)-\liminf_n\operatorname{Re}a_n(f-f_{m_j},g)(\omega,y)\leq
$$
$$
\leq 2\sup_n a_n(|f-f_{m_j}|,|g|)(\omega,y)\leq 2C \sup_n\frac 1n \sum_{k=0}^{n-1}|f-f_{m_j}|(\tau^k(\omega))\leq \frac {2C}j.
$$
Therefore $\Delta(\omega,y)=0$. Similarly,
$$
\limsup_n\operatorname{Im}a_n(f,g)(\omega,y)=\liminf_n\operatorname{Im}a_n(f,g)(\omega,y),
$$
and we conclude that the averages (\ref{eq51}) converge $\nu-$a.e. and all $\omega\in \Om_f$.

If $\mu$ is an arbitrary measure, we observe that, since $f\in \cal L^1(\Om)$, the restriction of $\mu$ on
the set $\{\omega \in \Om: f(\tau^k(\omega))\neq 0\}$
is $\sigma-$finite for each $k$, which reduces the argument to the case of $\sigma-$finite measure space $(\Om, \mu)$.
\end{proof}

Now we extend Theorem \ref {t52} to $\cal R_\mu=\cal R_\mu(\Om)$.
\begin{teo}\label{t53}
Let $(\Om, \mu)$ be a measure space, $\tau: \Om \to \Om$ be a m.p.t., and let $f\in \cal R_\mu$.
Then there exists a set $\Om_f\su \Om$ with $\mu(\Om \sm \Om_f)=0$ such that for any finite measure space $(Y,\nu)$, any
measure preserving transformation $\phi: Y\to Y$, and any $g\in \cal L^\ii(Y)$ the averages (\ref{eq51}) converge $\nu-$a.e.
for all $\omega \in \Om_f$.
\end{teo}
\begin{proof}
Due to Proposition \ref{p22}, given a natural $m$, there exists $f_m\in \cal L^1(\Om)$ and $h_m\in \cal L^\infty(\Om)$ such that
$f=f_m+h_m$ and $\| h_m\|_\infty \leq \frac 1m$. Then there is $\Om_0 \su \Om$ such that $\mu(\Om \sm \Om_0)=0$ and
$h_m(\omega)\leq \frac 1m$ for all $m$ and $\omega\in \Om_0$.

 By Theorem \ref{t52}, as $\{ f_m\}_{m=1}^\infty \su \cal L^1(\Om)$, for every $m$ there is a set $\Om_m \su \Om$ with
$\mu(\Om \sm \Om_m)=0$ such that for every $(Y,\nu,\phi)$ and $g\in \cal L^1(Y)$ the averages
\begin{equation}\label{eq53}
a_n(f_m,g)(\omega,y)=\frac 1n\sum_{k=0}^{n-1}f_m(\tau^k(\omega))g(\phi^k(y))
\end{equation}
converge $\nu-$a.e. for all $\omega\in \Om_m$. Therefore, if $\Om_f=\cap_{m=0}^\infty \Om_m$,
then $\mu(\Om \sm \Om_f)=0$, $h_m(\omega)\leq \frac 1m$ for all $m$ and $\omega\in \Om_f$,
and for every $(Y,\nu,\phi)$ and $g\in L^1(Y)$, the averages (\ref{eq53}) converge $\nu-$a.e. for all $m$ and $\omega\in \Om_f$.

Fix $\omega\in \Om_f$, $(Y,\nu,\phi)$, $g\in \cal L^1(Y,\nu)$ and show that the averages (\ref{eq51}) converge $\nu-$a.e.
As the averages (\ref{eq53}) converge $\nu-$a.e. for each $m$, there is a set $Y_1\su Y$ with $\nu(Y\sm Y_1)=0$ such that
the sequence (\ref{eq53}) converges for every $m$ and $y\in Y_1$. Also, since the averages
$$
\frac 1n\sum_{k=0}^{n-1}|g|(\phi^k(y))
$$
converge $\nu-$a.e., there is a set $Y_2\su Y$ such that $\nu(Y\sm Y_2)=0$ and the sequence
$\frac 1n\sum_{k=0}^{n-1}|g|(\phi^k(y))$ converges for all $y\in Y_2$. Then, letting $Y_0=Y_1\cap Y_2$, we conclude that
$\nu(Y\sm Y_0)=0$,  \ $\frac 1n\sum_{k=0}^{n-1}|g|(\phi^ky)<\infty$, and the sequence (\ref{eq53}) converges for all $m$ and $y\in Y_0$.

Now, if $y\in Y_0$, we have
$$
\liminf_n a_n(f_m,g)(\omega,y)=\limsup_n a_n(f_m,g)(\omega,y),
$$
which implies that, for every $m$,
$$
\Delta(\omega)=\limsup_n a_n(f,g)(\omega,y)-\liminf_na_n(f,g)(\omega,y)=
$$
$$
=\limsup_n a_n(h_m,g)(\omega,y)-\liminf_na_n(h_m,g)(\omega,y)\leq
$$
$$
\leq 2\sup_n \frac 1n\sum_{k=0}^{n-1}|h_m(\tau^k(\omega))|\cdot |g(\phi^k(y))|\leq \frac 2m\sup_n \frac 1n\sum_{k=0}^{n-1}|g|(\phi^k(y)).
$$
Therefore $\Delta(\omega)=0$ for every $y\in Y_0$, that is, the averages (\ref{eq51})
converge $\nu-$a.e.
\end{proof}

Letting in Theorem \ref{t53} \ $Y=\mathbb C_1=\{ y\in \mathbb C: |y|=1\}$ with Lebesgue measure
$\nu$, $\phi_\lb(y)=\lb y$, $y\in Y$,
for a given $\lb \in Y$, and $g(y)=y$ whenever $y\in Y$, we obtain Wiener-Wintner theorem for $\cal R_\mu$.

\begin{teo}\label{t54} Let $(\Om, \mu)$ be a measure space,
$\tau: \Om \to \Om$ be a m.p.t. If $f\in \cal R_\mu$, then
there is a set $\Om_f\su \Om$ with $\mu(\Om \sm \Om_f)=0$ such that the sequence
$$
\frac 1n \sum_{k=0}^{n-1} \lb^kf(\tau^k(\omega))
$$
converges for all $\omega\in \Om_f$ and $\lb\in \mathbb C_1$.
\end{teo}

Let $P(k)=\sum_{j=1}^s z_j\lb_j^k, \ k=0,1,2,\dots$ be a trigonometric polynomial (see Section 5).
 Then, by linearity, Theorem \ref{t54} implies the following.

\begin{cor}\label{c51}
If $(\Om,\mu)$ is a measure space,  $\tau: \Om \to \Om$ is a m.p.t. and  $f\in \cal R_\mu$, then there is
a set $\Om_f\su \Om$ with $\mu(\Om \sm \Om_f)=0$ such that the sequence
$$
a_n(\{P(k)\},f)(\omega)=\frac 1n\sum_{k=0}^{n-1}P(k)f(\tau^k(\omega))
$$
converges for every $\omega \in \Om_f$ and any trigonometric polynomial $\{P(k)\}$.
\end{cor}

We will need the following.

\begin{pro}\label{p51}
If $(\Om,\mu)$ is a measure space,  $\tau: \Om \to \Om$ is a m.p.t. and  $f\in \cal L^1(\Om)\cap \cal L^\ii(\Om)$,
then there exists a set $\Om_f\su \Om$ with $\mu(\Om \sm \Om_f)=0$ such that the sequence
$\frac 1n \sum_{k=0}^{n-1}b_kf(\tau^k(\omega))$ converges for every $\omega\in \Om_f$ and any bounded Besicovitch sequence $\{ b_k\}$.
\end{pro}
\begin{proof}
By Corollary \ref{c51}, there exists a set $\Om_{f,1}\su \Om$, $\mu(\Om\sm \Om_{f,1})=0$, such that the sequence
$\frac 1n\sum_{k=0}^{n-1}P(k)f(\tau^k\omega)$ converges for every $\omega\in \Om_{f,1}$ and any trigonometric polynomial $\{P(k)\}$.
Also, since $f\in \cal L^\ii(\Om)$, there is a set $\Om_{f,2}\su \Om$, $\mu(\Om\sm \Om_{f,2})=0$,
such that $|f(\tau^k\om)|\leq \| f\|_\ii$ for every
$k$ and $\om\in \Om_{f,2}$. If we set $\Om_f=\Om_{f,1}\cap \Om_{f,2}$, then $\mu(\Om\sm \Om_f)=0$.

Now, let $\om\in \Om_f$, and let $\{b_k\}$ be a Besicovitch sequence. Fix $\ep>0$, and choose a trigonometric polynomial
$P(k)$ to satisfy condition (\ref{eq041}). Then we have
$$
\Delta(\omega)=\limsup_n\operatorname{Re}a_n(\{b_k\},f)(\om)-\liminf_n\operatorname{Re}a_n(\{b_k\},f)(\om)=
$$
$$
=\limsup_n\operatorname{Re}a_n(\{b_k-P(k)\},f)(\om)-\liminf_n\operatorname{Re}a_n(\{b_k-P(k)\},f)(\om)\leq
$$
$$
\leq 2\| f\|_\ii\sup_n\frac 1n \sum_{k=0}^{n-1}|b_k-P(k)|<2\| f\|_\ii\ep
$$
for all sufficiently large $n$. Therefore $\Delta(\om)=0$, and we conclude that the sequence
$\operatorname{Re}\frac 1n \sum_{k=0}^{n-1}b_kf(\tau^k\om)$ converges.
Similarly, we obtain convergence of the sequence $\operatorname{Im} \frac 1n \sum_{k=0}^{n-1}b_kf(\tau^k\om)$,
which completes the proof.
\end{proof}

\begin{teo}\label{t56}
Let $(\Om,\mu)$ be a measure space. If $f\in \cal L^1(\Om)$, then there exists a set $\Om_f\su \Om$
with $\mu(\Om\sm \Om_f)=0$, such that the sequence
\begin{equation}\label{eq61}
a_n(\{b_k\},f)(\om)=\frac 1n\sum_{k=0}^{n-1}b_kf(\tau^k\om)
\end{equation}
converges for every $\om\in \Om_f$ and any bounded Besicovitch sequence $\{ b_k\}$.
\end{teo}
\begin{proof}
Let a sequence $\{ f_m\}\su \cal L^1(\Om)\cap \cal L^\ii(\Om)$ be such that $\| f-f_m\|_1\to 0$.
As in proof of Theorem \ref{t52}, we
construct a subsequence $\{ f_{m_j}\}$ and a set $\Om_0\su \Om$ with $\mu(\Om\sm \Om_0)=0$ such that
$$
\sup_n\frac 1n \sum_{k=0}^{n-1}|f-f_{m_j}|(\tau^k\om)\leq \frac 1j \text{ \ for all \ } j \text{ \ and \ } \om\in \Om_0.
$$
By Proposition \ref{p51}, given $j$, there is $\Om_j\su \Om$ with $\mu(\Om\sm \Om_j)=0$ such that the sequence
$\frac 1n\sum_{k=0}^{n-1}b_kf_{m_j}(\tau^k\om)$ converges for every $\om\in \Om_j$ and any Besicovitch sequence $\{ b_k\}$.

If we set $\Om_f=\cap_{j=1}^\ii \Om_j \cap \Om_0$, then $\mu(\Om\sm \Om_f)=0$, and for any $\om\in \Om_f$ and any bounded
Besicovitch sequence $\{ b_k\}$ such that $\sup_k|b_k|\leq C$ we have
$$
\Delta(\omega)=\limsup_n\operatorname{Re}a_n(\{b_k\},f)(\om)-\liminf_n\operatorname{Re}a_n(\{b_k\},f)(\om)=
$$
$$
=\limsup_n\operatorname{Re}a_n(\{b_k\},f-f_{m_j})(\om)-\liminf_n\operatorname{Re}a_n(\{b_k\},f-f_{m_j})(\om)\leq
$$
$$
\leq 2\sup_n\frac 1n \sum_{k=0}^{n-1}|b_k|\cdot |f-f_{m_j}|(\tau^k\om)\leq \frac {2C}j.
$$
Therefore $\Delta(\omega)=0$, hence the sequence $\operatorname{Re}\frac 1n\sum_{k=0}^{n-1}b_kf(\tau^k\om)$ converges.
Similarly, we derive convergence of the sequence $\operatorname{Im}\frac 1n\sum_{k=0}^{n-1}b_kf(\tau^k\om)$, and the
proof is complete.
\end{proof}

Taking into account that the sequence $\{ b_k\}$
is bounded, we obtain, as in the proof of Theorem \ref{t53}, the following extension of Wiener-Wintner theorem.

\begin{teo}\label{t57}
Let $(\Om,\mu)$ be a measure space, and let $\tau:\Om\to \Om$ be a m.p.t. Given $f\in \cal R_\mu$,
then there exists a set $\Om_f\su \Om$ with $\mu(\Om\sm \Om _f)=0$ such that the sequence (\ref{eq61})
converges for every $\om\in \Om_f$ and every bounded Besicovitch sequence $\{ b_k\}$.
\end{teo}

 Finally, in view of Proposition \ref{p21}, we have the following.
\begin{teo}\label{t58}
Let $(\Om,\mu)$ be an infinite measure space, and let $\tau:\Om\to \Om$ be a m.p.t.
Assume that $E=E(\Om)$ is a fully symmetric space such that $\mathbf 1\notin E$. Then for every $f\in E$
there exists a set $\Om_f\su \Om$ with $\mu(\Om\sm \Om _f)=0$ such that the sequence (\ref{eq61})
converges for every $\om\in \Om_f$ and every bounded Besicovitch sequence $\{ b_k\}$.
\end{teo}

\end{document}